\documentclass[a4paper,9pt]{article}
\usepackage{amssymb,amsmath,amsthm,graphicx,fancyhdr}

\usepackage{color}

\newtheorem{Def}{Definition}
\newtheorem{Thm}{Theorem}
\newtheorem{Ass}{Assumption}

\newtheorem{Lem}{Lemma}
\newtheorem{Rem}{Remark}
\newtheorem{Cor}{Corollary}

\newcommand{\R}{\ensuremath{\mathbb{R}}}

\newcommand{\al}{\ensuremath{\alpha}}

\newcommand{\de}{\ensuremath{\delta}}
\newcommand{\ep}{\ensuremath{\varepsilon}}

\newcommand{\om}{\ensuremath{\omega}}
\newcommand{\la}{\ensuremath{\lambda}}

\newcommand{\La}{\ensuremath{\Lambda}}
\newcommand{\De}{\ensuremath{\Delta}}
\numberwithin{equation}{section}

\newcommand{\set}[2]{\ensuremath{\left\{ #1 \ \big| \ #2\right\}}}

\newcommand{\half}{\ensuremath{\frac{1}{2}}}

\newcommand{\pd}{\ensuremath{\partial}}

\newcommand{\lr}[2]{\ensuremath{\left\langle #1 , #2 \right\rangle }}
\newcommand{\pom}{\ensuremath{\phi_{\om}}}
\newcommand{\lpom}[1]{\ensuremath{\phi_{\om+ #1 }}}
\newcommand{\ppom}{\ensuremath{\pd_{\om}\phi_{\om}}}
\newcommand{\lppom}[1]{\ensuremath{\pd_{\om}\phi_{\om+ #1 }}}
\newcommand{\Tpom}{\ensuremath{T'(0)\phi_{\om}}}

\newcommand{\Som}{\ensuremath{S_{\om}}}
\newcommand{\Hom}{\ensuremath{S_{\om}''(\pom)}}

\title
{Stability of bound states of Hamiltonian PDEs in the degenerate cases}
\author{Masaya Maeda${}^{\dagger}$}

\date{}

\begin{document}

\maketitle

\vskip-5mm
\centerline{${}^{\dagger}$Institute of Mathematics, 
Tohoku University, }
\centerline{Sendai, 980-8578, Japan}

\begin{abstract}
We consider a Hamiltonian systems which is invariant under a one-parameter unitary group.
We give a criterion for the stability and instability of bound states for the degenerate case.
We apply our theorem to the single power nonlinear Klein-Gordon equation and the double power nonlinear Schr\"odinger equation.
\end{abstract}

\section{Introduction}
In this paper, following a celebrated paper \cite{GSS} by Grillakis, Shatah and Strauss, we consider the abstract Hamiltonian system of the form
\begin{eqnarray}\label{eq:intro}
\frac{du}{dt}(t)=JE'(u),
\end{eqnarray}
where $E$ is the energy functional on a real Hilbert space $X$, and $J:X^*\rightarrow Y^*$ is a skew-symmetric operator.
Here, $Y$ is another real Hilbert space and $u\in C({\cal I}, X)\cap C^1({\cal I},Y^*)$ for some interval ${\cal I}$.
Equation (\ref{eq:intro}) can be considered as a generalization of nonlinear Schr\"odinger equations (NLS) and nonlinear Klein-Gordon equations (NLKG).
We assume that $E$ is invariant under a one-parameter unitary group $\{T(s)\}_{s\in\R}$.
We consider the stability and instability of bound state solutions $T(\om t)\phi_{\om}$ of (\ref{eq:intro}), where $\om\in \R$ and $\phi_{\om}\in X$.
We assume that the linearized Hamiltonian
\begin{eqnarray*}
S_{\om}''(\pom):=E''(\phi_{\om})-\om Q''(\phi_{\om})
\end{eqnarray*} 
has one negative eigenvalue, where $Q$ is the invariant quantity which comes out from the Noether's principal due to the symmetry $T(s)$.

In \cite{GSS}, it is proven that if $d''(\om)>0$ (resp.\ $<0$), then the bound state $T(\om t)\pom$ is stable (resp.\ unstable), where
\begin{eqnarray*}
d(\om):=E(\pom)-\om Q(\pom).
\end{eqnarray*}
Further, Theorem 2 of \cite{GSS} claims that ``bound states are stable if and only if $d$ is strictly convex in a neighborhood of $\om$''.
However, as pointed out by Comech and Pelinovsky \cite{CP03}, their argument seems to be not correct for the case $d''(\om)=0$.
Our aim of this paper is to recover this criterion, i.e.\ investigate the stability and instability for the case $d''(\om)=0$.

For the case $d''(\om)=0$, Comech and Pelinovsky \cite{CP03} showed that if $d''(\tilde{\om})\leq 0$ in a one-sided open neighborhood of $\om$, then the bound state $T(\om t)\phi_{\om}$ is unstable.
Their proof is based on the observation that in the case $d''(\om)=0$, the linearized operator $J\Hom$ has a degenerate zero eigenvalue which leads to a polynomial growth of perturbations.
They show the instability by considering (\ref{eq:intro}) as a perturbation of the linearized equation around $\phi_{\om}$.
Recently, Ohta \cite{Ohta11} gave another proof for the instability of bound states for the case $d''(\om)=0$, $d'''(\om)\neq 0$.
His proof is based on \cite{GSS} and \cite{M10} which uses a Lyapunov functional to ``push out'' the solutions from the neighborhood of the bound states.
However, \cite{Ohta11} assumes $T'(0)=J$ and this assumption prevent his result to apply to the NLKG equations.

In this paper, we follow the work of \cite{GSS}, \cite{M10} and \cite{Ohta11} and extend the results of \cite{GSS} and \cite{Ohta11}.
We show that, if $d''(\om)$ is strictly convex in a neighborhood of $\om$, then the bound is stable and if $d(\tilde{\om})-d(\om)-(\tilde{\om}-\om) d'(\om)<0$ in $\om<\tilde{\om}<\om+\ep$ or $\om-\ep<\tilde{\om}<\om$ for some $\ep>0$, then the bound state is unstable.
For the meaning of assumption ``$d(\tilde{\om})-d(\om)-(\tilde{\om}-\om) d'(\om)<0$'', consider the following three conditions.

\begin{enumerate}
\item[$(\mathrm{A})$]
$\exists \ep>0$ s.t.\ $\forall\la \in (0,\ep)$ (resp. $\forall\la\in (-\ep,0)$), $d''(\om+\la)<0$.
\item[$(\mathrm{B})$]
$\exists \ep>0$ s.t.\ $\forall\la \in (0,\ep)$ ($\forall\la\in (-\ep,0)$), $d(\om+\la)-d(\om)+\la d'(\om)<0$.
\item[$(\mathrm{C})$]
$\exists \{\la_n\}$ s.t.\ $\la_n\rightarrow 0$ and $d''(\om+\la_n)<0$.
\end{enumerate}
Then, we have (A)$\Rightarrow $(B)$\Rightarrow $(C) and (C) is equivalent to ``$d$ is not convex in the neighborhood of $\om$''.
Therefore, our assumption, which is condition (B), do not cover the case ``$d$ is not convex in the neighborhood of $\om$'', but the gap can considered to be small.
If $d''(\om)=0$ and $d'''(\om)\neq 0$, then we have (A).
So, our result covers the result of \cite{Ohta11}.
The only natural case which we cannot treat in our theorem is the case $d$ is linear in a one-sided open neighborhood of $\om$.
In this sense we have almost proved the criterion ``bound states are stable if and only if $d(\om)$ is strictly convex''.

The proof is based on a purely variational argument.
We note that our result almost covers the result of \cite{CP03} but not completely.
The case $d$ is linear in the neighborhood of $\om$ is excluded by our theorem, which this case can be covered by \cite{CP03} .
However, our proof requires less regularity for $E$, which is $E\in C^2$ and does not need an assumption for nonlinearity like Assumption 2.10 (b), (c) of \cite{CP03}.

We give an application of our theorem for the single power NLKG equations and double power nonlinear Schr\"odinger equations.
For the one dimensional NLKG with $|u|^{p-1}u$, $1<p<2$, our result seems to be new.
Further, we remark our result covers all dimensions in a unified way.

This paper is organized as follows:
In section \ref{sec:nr}, we formulate our assumptions and the main results in a precise manner.
In section \ref{sec:p}, we prepare some notations and lemmas for the proof of the main results.
In particular, we construct a curve $\Psi(\la)$ on the hyper-surface ${\cal M}:=\{Q(u)=Q(\pom)\}$, which crosses the set of the bound state.
Then, we calculate $S_{\om}(\Psi(\la))$ and $P(\Psi(\la))$, where $P$ is a functional which we will use for the instability.
This curve $\Psi(\la)$ corresponds to the degenerate direction of the energy functional $E$ on the hyper-surface ${\cal M}$.
In section \ref{sec:pt}, we prove the main results.
We calculate $S_{\om}$ and $P$ for general $u$ in a neighborhood of $\pom$ under some restrictions on the value of $S_{\om}$.
The restrictions give us a good estimate for the ``nondegenerate'' directions and enables us to use the results of section {\ref{sec:p}}.
In section \ref{sec:e}, we give an applications of the main theorem for NLKG and NLS equations.

\section{Notation and main results}\label{sec:nr}

Let $X$, $Y$ and $H$ be a real Hilbert spaces such that
\begin{eqnarray*}
X\hookrightarrow H \simeq H^* \hookrightarrow X^*,\ Y\hookrightarrow H \simeq H^* \hookrightarrow Y^*,
\end{eqnarray*}
where all the embeddings are densely continuous.
We identify $H$ with $H^*$ naturally.
We denote both the inner product of  $H$, the coupling between $X$ and $X^*$ and the coupling between $Y$ and $Y^*$ by $\lr{\cdot}{\cdot}$.
The norms of $X$ and $H$ are denoted as $||\cdot||_X$ and $||\cdot||_H$, respectively.
Let $J:H\rightarrow H$ be a skew-symmetric operator in such a sense that
\begin{eqnarray*}
\lr{Ju}{v}=-\lr{u}{Jv},\ u,v\in H.
\end{eqnarray*}
Further, we assume $\left.J\right|_X:X\rightarrow Y$ and $\left. J\right|_Y:Y\rightarrow X$ are bijective and bounded.
The operator $J$ can be naturally extended to $\tilde{J}:X^*\rightarrow Y^*$ by
\begin{eqnarray*}
\lr{\tilde{J}u}{v}:=-\lr{u}{Jv},\ u\in X^*,\ v\in Y.
\end{eqnarray*}
Let $T(s)$ be a one parameter unitary group on $X$ and let $T'(0)$ is the generator of $T(s)$.
We denote the domain of $T'(0)$ by $D(T'(0))\subset X$.
As $J$, we can naturally extend $T(s)$ to $\tilde{T}(s):X^*\rightarrow X^*$ by
\begin{eqnarray*}
\lr{\tilde{T}(s)u}{v}:=\lr{u}{T(-s)v},\ u\in X^*,\ v\in X.
\end{eqnarray*}
We assume $\tilde{T}(s)(Y)\subset Y$ for all $s\in\R$.
For simplicity, we just denote $\tilde{T}(s)$ as $T(s)$.
We further assume that $J$ and $T(s)$
 commute.

Let $E\in C^2(X,\R)$.
We consider the following Hamiltonian PDE.
\begin{eqnarray}\label{eq:hpde}
\frac{du}{dt}(t)=\tilde{J}E'(u(t)),
\end{eqnarray}
where $E'$ is the Fr\'echet derivative of $E$.
We say that $u(t)$ is a solution of (\ref{eq:hpde}) in an interval ${\cal I}$ if $u\in C({\cal I}, X)\cap C^1({\cal I},Y^*)$ and satisfies (\ref{eq:hpde}) in $Y^*$ for all $t\in{\cal I}$.
We assume that $E$ is invariant under $T$, that is,
\begin{eqnarray*}
E(T(s)u)=E(u),\ s\in\R,\ u\in X.
\end{eqnarray*}

We assume that there is a bounded operator $B:X\rightarrow X^*$ such that $B^*=B$ and the operator $B$ is an extension of $J^{-1}T'(0)$.
We define $Q:X\rightarrow \R$ by 
\begin{eqnarray}\label{eq:defQ}
Q(u):=\half\lr{Bu}{u},\ u\in X.
\end{eqnarray}
Then, we have $Q(T(s)u)=Q(u)$ for $u\in X$.
Indeed, for $u\in D(T'(0))$, we have
\begin{eqnarray*}
\frac{d}{ds}Q(T(s)u)&=&\lr{BT(s)u}{T'(0)T(s)u}\\
&=&\lr{BT(s)u}{JBT(s)u}=0.
\end{eqnarray*}
For general $u\in X$, we only have to take a sequence $u_n\in D(T'(0))$, $u_n\rightarrow u$ in $X$.
Further, formally $Q$ conserves under the flow of (\ref{eq:hpde}).
Indeed,
\begin{eqnarray*}
\frac{d}{dt}Q(u(t))&=&\lr{Bu(t)}{J^{-1}E'(u(t))}\\
&=&\lr{T'(0)u(t)}{E'(u(t))}\\
&=&\left.\frac{d}{ds}\right|_{s=0}E(T(s)u(t))=0.
\end{eqnarray*}


We now assume that the Cauchy problem of (\ref{eq:hpde}) is well-posed in $X$.

\begin{Ass}[Existence of solutions]\label{ass:sol}
Let $\mu>0$.
Then, there exists $T(\mu)>0$ such that for all $u_0\in X$ with $||u_0||_X\leq \mu$, we have a solution $u$ of $(\ref{eq:hpde})$ in $[0,T(\mu))$ with $u(0)=u_0$.
Further, $u$ satisfies $E(u(t))=E(u_0)$ and $Q(u(t))=Q(u_0)$ for $t\in (0,T(\mu))$.
\end{Ass}

We next define the bound state, which is an stationary solution modulo symmetry $T(s)$.

\begin{Def}[Bound state]
By a bound state we mean a solution of $(\ref{eq:hpde})$ in $\R$ with the form
\begin{eqnarray*}
u(t)=T(\om t)\phi,
\end{eqnarray*}
where $\om\in\R$ and $\phi\in X$.
\end{Def}

\begin{Rem}
If $u(t)=T(\om t)\phi$ is a bound state, then it satisfies
\begin{eqnarray*}
\om T(\om t)T'(0)\phi=JE'(T(\om t)\phi).
\end{eqnarray*}
Thus, by $E'(T(s)u)=T(s)E'(u)$ and the definition of $Q$, we have
\begin{eqnarray}\label{eq:condbddst}
E'(\phi)-\om Q'(\phi)=0.
\end{eqnarray}
On the other hand, if $\phi\in X$ satisfies $(\ref{eq:condbddst})$, then $T(\om t)\phi$ is a bound state.
\end{Rem}

\begin{Def}[Stability of bound states]
We say the bound state $T(\om t)\phi$ is stable if for all $\ep>0$ there exists $\de>0$ with the following property.
If \\$||u_0-\phi||_X<\de$ and $u(t)$ is a solution of $(\ref{eq:hpde})$ given in Assumption $\ref{ass:sol}$, then $u(t)$ can be continued to a solution in $[0,\infty)$ and
\begin{eqnarray*}
\sup_{0<t}\inf_{s\in\R}||u(t)-T(s)\phi||_X<\ep.
\end{eqnarray*}
Otherwise the bound state $T(\om t)\phi$ is said to be unstable.
\end{Def}

\begin{Ass}[Existence of bound states]\label{ass:bddst}
Let $\om_1<\om_2$.
We assume that there exists a $C^3$ map $(\om_1, \om_2)\rightarrow X$, $\om\mapsto \pom$ such that
\begin{enumerate}
\item[$(\mathrm{i})$]
$T(\om t)\pom$ is a bound state.
\item[$(\mathrm{ii})$]
$\pom\in D(T'(0)^3)$, $\ppom\in D(T'(0)^2)$, $\pd_{\om}^2\pom\in D(T'(0))$,\\ $\Tpom, \ppom, T'(0)\ppom, \pd_{\om}^2\pom\in Y$.
\item[$(\mathrm{iii})$]
$T'(0)\pom\neq 0$, $\ppom\neq 0$ and $\lr{T'(0)\pom}{\ppom}=0$.

\end{enumerate}
\end{Ass}

\begin{Rem}
By the fact that $\Tpom\in Y$, we have $B\pom=J^{-1}\Tpom\in X$.
\end{Rem}
\begin{Rem}
$\lr{T'(0)\pom}{\ppom}=0$ is actually not an assumption.
Indeed, suppose $\om\mapsto \pom$ does not satisfy $\lr{T'(0)\pom}{\ppom}=0$.
Then, set $\tilde{\phi}_{\om}=T(s(\om))\pom$, where
\begin{eqnarray*}
s(\om)=-\int_0^{\om}\frac{\lr{T'(0)\phi_{\mu}}{\pd_{\mu}\phi_{\mu}}}{||T'(0)\phi_{\mu}||_H^2}\,d\mu.
\end{eqnarray*}
Then, $\tilde{\phi}_{\om}$ satisfies $\lr{T'(0)\tilde{\pom}}{\pd_{\om}\tilde{\pom}}=0$.
\end{Rem}

Set
\begin{eqnarray}
S_{\om}(u)&:=&E(u)-\om Q(u),\ u\in X,\nonumber\\
d(\om)&:=&S_{\om}(\pom),\label{eq:defd}
\end{eqnarray}
where $\pom$ is given in Assumption \ref{ass:bddst}.

\begin{Rem}
Condition $(\ref{eq:condbddst})$ is equivalent to $S_{\om}'(\phi)=0$.
\end{Rem}

We further assume that the linearized Hamiltonian $\Hom$ satisfies the following spectral condition.

\begin{Ass}[Spectral conditions for the bound states]\label{ass:spec}
For $\om\in (\om_1,\om_2)$, we assume the following.
\begin{enumerate}
\item[$(\mathrm{i})$]
$\mathrm{ker}\Hom=\mathrm{span}\{ T'(0)\pom \}$
\item[$(\mathrm{ii})$]
$\Hom$ has only one simple negative eigenvalue $-\mu<0$.
\item[$(\mathrm{iii})$]
$\inf\set{s>0}{s\in\sigma(\Hom)}>0$,
\end{enumerate}
where $\sigma(\Hom)\subset\R$ is the spectrum of $\Hom$.
\end{Ass}

Grillakis, Shatah and Strauss \cite{GSS} proved the following theorem.

\begin{Thm}\label{thm:gss}
Let Assumptions $\ref{ass:sol}$, $\ref{ass:bddst}$ and $\ref{ass:spec}$ be satisfied.
Then, if $d''(\om)>0$, the bound state $T(\om t)\pom$ is stable and if $d''(\om)<0$, the bound state $T(\om t)\pom$ is unstable.
\end{Thm}

We investigate the case $d''(\om)=0$.

We denote $f(\la)\sim g(\la)$ if $f$ and $g$ satisfies
\begin{eqnarray}\label{eq:defsim}
0<\liminf_{|\la|\rightarrow 0}f(\la)/g(\la)\leq \limsup_{|\la|\rightarrow 0}f(\la)/g(\la)<\infty.
\end{eqnarray}
We assume
\begin{eqnarray}\label{eq:comparable}
d(\om+\la)-d(\om)-\la d'(\om)\sim \la(d'(\om+\la)-d'(\om)).
\end{eqnarray}
This is a technical assumption which we need in the proof.

\begin{Rem}
If $d\in C^n$ and $d^{(m)}(\om)\neq 0$ for some $2<m\leq n$, then the assumption $(\ref{eq:comparable})$ is satisfied.
Let $d(\om+\la)=e^{-1/|\la|}$, then $d$ does not satisfy $(\ref{eq:comparable})$.
However, this assumption seems to be natural.
\end{Rem}

We now state our main results.

\begin{Thm}\label{thm:stability}
Let Assumptions $\ref{ass:sol}$, $\ref{ass:bddst}$, $\ref{ass:spec}$ and $(\ref{eq:comparable})$ be satisfied.
Assume that $d$ is strictly convex in an open neighborhood of $\om$.
Then $T(\om t)\pom$ is stable.
\end{Thm}

\begin{Rem}
For Theorem $\ref{thm:stability}$, we can remove the condition $J|_X$, $J|_Y$ are bijective and bounded.
Further, we only need $\om\mapsto \pom$ to be $C^2$.
We only use these conditions for Theorem $\ref{thm:instability}$ below, which is concerned with the instability.
Therefore, we can treat the case $J=\pd_x$, which appears for KdV type equations and BBM type equations.
\end{Rem}

\begin{Thm}\label{thm:instability}
Let Assumptions $\ref{ass:sol}$, $\ref{ass:bddst}$, $\ref{ass:spec}$ and $(\ref{eq:comparable})$ be satisfied.
Assume there exists $\ep>0$ such that $d(\om+\la)-d(\om)-\la d'(\om)<0$ in $0<\la<\ep$ or $-\ep<\la<0$.
Further, assume $\lr{\phi_{\om+\la}}{J^{-1}\pd_{\om}^2\phi_{\om+\la}}=0$.
Then $T(\om t)\pom$ is unstable.
\end{Thm}

\begin{Rem}
If $d''(\om)>0$ $($resp.\ $<0$$)$, then the assumption of Theorem $\ref{thm:stability}$ $($resp.\ Theorem $\ref{thm:instability}$$)$ is satisfied.
Therefore, Theorems $\ref{thm:stability}$ and $\ref{thm:instability}$ are extension of Theorem $\ref{thm:gss}$.
\end{Rem}

\begin{Rem}
The assumption $\lr{\phi_{\om+\la}}{J^{-1}\pd_{\om}^2\phi_{\om+\la}}=0$ is technical.
However, for the NLS and NLKG cases, this is satisfied when as far as the real valued standing waves are concerned.
\end{Rem}

\begin{Cor}\label{cor:stability}
Let Assumptions $\ref{ass:sol}$, $\ref{ass:bddst}$ and $\ref{ass:spec}$ be satisfied.
Let $n\geq 4$ be an even integer.
Assume that $d\in C^n$ in an open neighborhood of $\om$ and assume
\begin{eqnarray*}
d''(\om)=\cdots=d^{(n-1)}(\om)=0,\ d^{(n)}(\om)>0.
\end{eqnarray*}
Then $T(\om t)\phi_{\om}$ is stable.
\end{Cor}

\begin{Cor}\label{cor:instability}
Let Assumptions $\ref{ass:sol}$, $\ref{ass:bddst}$ and $\ref{ass:spec}$ be satisfied.
Further, assume there exists $\ep>0$ such that $\lr{\phi_{\om+\la}}{J^{-1}\pd_{\om}^2\phi_{\om+\la}}=0$ for $|\la|<\ep$.
Let $n\geq 3$ be an integer.
Assume that $d \in C^n$ in a open neighborhood of $\om$ and
\begin{eqnarray*}
d''(\om)=\cdots=d^{(n-1)}(\om)=0,\\ d^{(n)}(\om)<0\ (n:even),\ d^{(n)}(\om)\neq 0\ (n:odd).
\end{eqnarray*}
Then $T(\om t)\phi_{\om}$ is unstable.
\end{Cor}

\section{Preliminaries}\label{sec:p}
In this section, we assume Assumptions \ref{ass:sol}, \ref{ass:bddst}, \ref{ass:spec}, (\ref{eq:comparable}) and $d''(\om)=0$.
Note that by differentiating (\ref{eq:defd}) with respect to $\om$, we have
\begin{eqnarray}
d'(\om)&=&S_{\om}'(\pom)-Q(\pom)=-Q(\pom)\label{eq:d'}\\
d''(\om)&=&-\lr{B\pom}{\ppom}.
\end{eqnarray}
Further, differentiating the equation $S_{\om}'(\pom)=0$ with respect to $\om$, we have
\begin{eqnarray}\label{eq:SB}
S_{\om}''(\pom)\ppom=B\pom.
\end{eqnarray}
We will use these relations in the following.
Set
\begin{eqnarray}
\eta_1(\la)&:=&d(\om+\la)-d(\om)-\la d'(\om)\label{eq:eta1}\\
\eta_2(\la)&:=&d'(\om+\la)-d'(\om),\label{eq:eta2}
\end{eqnarray}
Recall that in (\ref{eq:comparable}), we have assumed $\eta_1(\la)\sim \la\eta_2(\la)$.
Further, since we are assuming $d''(\om)=0$, we have $\eta_2(\la)=o(\la)$ as $\la\rightarrow 0$.

\begin{Lem}\label{lem:defsigma}
Let $\ep>0$ sufficiently small.
Then, there exists $\sigma(\la):(-\ep,\ep)\rightarrow \R$ such that $\sigma(\la)\sim\eta_2(\la)$ and
\begin{eqnarray}\label{eq:defsigma}
Q(\lpom{\la}+\sigma(\la)B\lpom{\la})=Q(\pom),
\end{eqnarray}
for $|\la|<\ep$, where we have used ``$\sim$'' in the sense of $(\ref{eq:defsim})$.
\end{Lem}

\begin{proof}
Set
\begin{eqnarray*}
F(\sigma,\la)=Q(\lpom{\la}+\sigma B\lpom{\la}).
\end{eqnarray*}
Then, $F(0,0)=Q(\pom)$ and $\left.\pd_{\sigma}F\right|_{\sigma=\la=0}(\sigma,\la)=||B\pom||_H^2\neq 0$.
Therefore, by the implicit function theorem, there exist $\ep>0$, $\de>0$ and $\sigma:(-\ep,\ep)\rightarrow (-\de,\de)$ such that $\sigma(\la)$ satisfies (\ref{eq:defsigma}) for $|\la|<\ep$.
Further, by (\ref{eq:defsigma}), we have
\begin{eqnarray*}
\sigma(\la)\left(||B\phi_{\om+\la}||_H^2+\sigma(\la)Q(B\phi_{\om+\la})\right)&=&-Q(\lpom{\la})+Q(\pom)\\
&=&d'(\om+\la)-d'(\om)\\
&=&\eta_2(\la),
\end{eqnarray*}
where we have used (\ref{eq:d'}) and (\ref{eq:eta2}).
Since
\begin{eqnarray*}
\sigma(\la)\left(||B\phi_{\om+\la}||_H^2+\sigma(\la)Q(B\phi_{\om+\la})\right)=\sigma(\la)(||B\pom||_H^2+o(1))\  \mathrm{as}\ \la\rightarrow 0,
\end{eqnarray*}
we have the conclusion.
\end{proof}

We now define a curve on the neighborhood of $\pom$.
Let $\ep>0$ as in Lemma \ref{lem:defsigma}.
For $|\la|<\ep$, set 
\begin{eqnarray*}
\Psi(\la):=\lpom{\la}+\sigma(\la)B\lpom{\la}.
\end{eqnarray*}
We next calculate the value of $S_{\om}(\Psi(\la))$.

\begin{Lem}\label{lem:S}
Let $\ep>0$ as Lemma $\ref{lem:defsigma}$.
Then for $|\la|<\ep$, we have
\begin{eqnarray*}
\Som(\Psi(\la))-\Som(\pom)=\eta_1(\la)+o(\eta_1(\la)),\ \la\rightarrow 0.
\end{eqnarray*}
\end{Lem}

\begin{proof}
By the definition of $S_{\om}$, we have $S_{\om}=S_{\om+\la}+\la Q$.
Using this and Taylor expansion, we have
\begin{eqnarray*}
S_{\om}(\Psi(\la))&=&S_{\om+\la}(\Psi(\la))+\la Q(\Psi(\la))\\
&=& S_{\om+\la}(\lpom{\la}+\sigma(\la)B\lpom{\la})+\la Q(\pom)\\
&=& S_{\om+\la}(\lpom{\la})+\la Q(\pom)+O(\sigma(\la)^2)\\
&=& d(\om+\la) -\la d'(\om) + o(\eta_1(\la)),\ \la\rightarrow 0.
\end{eqnarray*}
Here, we have used $Q(\Psi(\la))=Q(\pom)$ for the second equality, $S_{\om+\la}'(\lpom{\la})=0$ for the third equality and $\sigma(\la)=o(\la)$, $O(\la\sigma(\la))=O(\la\eta_2(\la))=O(\eta_1(\la))$ for the last equality.
Therefore, by (\ref{eq:eta1}), we have the conclusion.
\end{proof}

We define a tubular neighborhood of $\pom$.
Set
\begin{eqnarray*}
N_{\ep}&:=&\set{u\in X}{\inf_{s\in\R}||u-T(s)\pom||_X<\ep},\\
N_{\ep}^0&:=&\set{u\in N_{\ep}}{Q(u)=Q(\pom)}.
\end{eqnarray*}

\begin{Lem}\label{lem:modul}
Let $\ep>0$ sufficiently small.
Then for $u\in N_{\ep}$, there exists $\theta(u)$, $\La(u)$, $w(u)$ and $\al(u)$ such that
\begin{eqnarray*}
T(\theta(u))u=\Psi(\La(u))+w(u)+\al(u)B\lpom{\La(u)},
\end{eqnarray*}
and
\begin{eqnarray*}
\lr{w(u)}{T'(0)\lpom{\La(u)}}=\lr{w(u)}{\lppom{\La(u)}}=\lr{w(u)}{B\lpom{\La(u)}}=0.
\end{eqnarray*}
Further, $\La$ and $\theta$ are $C^2$.
\end{Lem}

\begin{proof}
Set
\begin{eqnarray*}
G(u,\theta,\La)=
\left(
  \begin{matrix}
        \lr{T(\theta)u-\Psi(\La)}{T'(0)\lpom{\La}} \cr
        \lr{T(\theta)u-\Psi(\La)}{\lppom{\La}} \cr
  \end{matrix}
\right).
\end{eqnarray*}
Then, we have $G(\pom,0,0)=0$ and
\begin{eqnarray}\label{eq:G'}
\frac{\pd G}{\pd (\theta, \La)}=(G_{ij}(u,\theta,\La))_{i,j=1,2},
\end{eqnarray}
where
\begin{eqnarray*}
G_{11}(u,\theta,\La)&=&\lr{T'(0)T(\theta)u}{T'(0)\lpom{\La}}, \\
G_{12}(u,\theta,\La)&=&-\lr{\pd_\la\Psi(\La)}{T'(0)\lpom{\La}}+\lr{T(\theta)u-\Psi(\La)}{T'(0)\lppom{\La}},\\
G_{21}(u,\theta,\La)&=&\lr{T'(0)T(\theta)u}{\lppom{\La}},\\
G_{22}(u,\theta,\La)&=&-\lr{\pd_\la\Psi(\La)}{\lppom{\La}}+\lr{T(\theta)u-\Psi(\La)}{\pd_{\om}^2\lpom{\La}}.
\end{eqnarray*}
Therefore,
\begin{eqnarray*}
\left.\frac{\pd G}{\pd (\theta, \La)}\right|_{u=\pom,\theta=\La=0}=
\left(
  \begin{matrix}
        ||T'(0)\pom||_H^2 & 0 \cr
        0 & -||\ppom||_H^2 \cr
  \end{matrix}
\right),
\end{eqnarray*}
is invertible.
Thus, there exist functionals $\theta(u)$ and $\La(u)$ defined in the neighborhood of $\pom$ such that $G(u,\theta(u),\La(u))=0$.
Since, $\om'\mapsto \phi_{\om'}$ is a $C^3$ map, we see that $G$ is $C^2$.
Therefore, we have  $\La$ and $\theta$ are $C^2$.
For $u\in N_{\ep}$, define $\theta(T(s)u)=\theta(u)-s$ and $\La(T(s)u)=\La(u)$.
Finally, define
\begin{eqnarray*}
\al(u)&=&\lr{T(\theta(u))u-\Psi(\La(u))}{B\lpom{\La(u)}}||B\lpom{\La(u)}||_H^{-2},\\
w(u)&=&T(\theta(u))u-\Psi(\La(u))-\al(u)B\lpom{\La(u)}.
\end{eqnarray*}
Therefore, we have the conclusion.
\end{proof}

Let $\ep>0$ as Lemma \ref{lem:modul}.
Set
\begin{eqnarray*}
M(u):=T(\theta(u))u,\ u\in N_{\ep}.
\end{eqnarray*}

\begin{Rem}
By the uniqueness of the solution of $G=0$, we have
\begin{eqnarray*}
\theta(\Psi(\la))=0,\ 
\la(\Psi(\la))=\la,\\
w(\Psi(\la))=0,\ 
\al(\Psi(\la))=0.\\
\end{eqnarray*}
\end{Rem}

We next show that the Fr\'echet derivatives of $\theta$ and $\La$ are in $Y$.

\begin{Lem}\label{lem:regularity}
Let $\ep>0$ sufficiently small.
Let $u\in N_{\ep}$.
Then, $\theta'(u)$, $\La'(u)\in Y$.
\end{Lem}

\begin{proof}
By differentiating $G(u,\theta(u),\La(u))=0$ with respect to $u$, we have
\begin{eqnarray}\label{eq:G'}
H(u)
\left(
  \begin{matrix}
        \theta'(u) \cr
        \La'(u) \cr
  \end{matrix}
\right)=
-\left(
  \begin{matrix}
        T(-\theta(u))T'(0)\lpom{\La(u)} \cr
        T(-\theta(u))\lppom{\La(u)} \cr
  \end{matrix}
\right),
\end{eqnarray}
where $H(u)=(G_{i,j}(u,\theta(u),\La(u)))_{i,j=1,2}$.
Since $H(u)$ is invertible in $N_{\ep}$ for sufficiently small $\ep>0$ and $T'(0)\lpom{\La(u)}\in Y$, $\lppom{\La(u)}\in Y$, we have the conclusion.
\end{proof}

\begin{Rem}
As the proof of Lemma $\ref{lem:regularity}$,
by differentiating $(\ref{eq:G'})$ with respect to $u$, we see that $\theta''(u)w\in Y$ and $\La''(u)w\in Y$ for $u\in N_{\ep}$ and $w\in X$.
\end{Rem}

Let $\ep>0$ sufficiently small.
We now introduce the following functionals $A$ and $P$ defined in $N_{\ep}$, which we use to show the instability theorem.
\begin{eqnarray*}
A(u)&:=&\lr{M(u)}{J^{-1}\lppom{\La(u)}},\\
P(u)&:=&\lr{S_{\om+\La(u)}'(u)}{JA'(u)}.
\end{eqnarray*}

\begin{Rem}\label{rem:regularity}
$A$ and $P$ are well-defined in $N_{\ep}$ for sufficiently small $\ep>0$.
Indeed,
\begin{eqnarray}\label{eq:A'}
A'(u)&=&J^{-1}T(-\theta(u))\lppom{\La(u)}+\lr{T'(0)M(u)}{J^{-1}\lppom{\La(u)}}\theta'(u)\nonumber\\&&+\lr{M(u)}{J^{-1}\pd_{\om}^2\phi_{\om+\La(u)}}\La'(u).
\end{eqnarray}
So, by Assumption $\ref{ass:bddst}$ and Lemma $\ref{lem:regularity}$, we have $A'(u)\in Y$.
Therefore, we have $JA'(u)\in X$.
Therefore, the definition of $P$ makes sense.
\end{Rem}

\begin{Rem}
Let $u$ be the solution of $(\ref{eq:hpde})$, then
\begin{eqnarray}\label{eq:AP}
\frac{d}{dt}A(u(t))=-P(u(t)).
\end{eqnarray}
Indeed, first, since $A(T(s)u)=A(u)$, for $u\in D(T'(0))$,
\begin{eqnarray*}
0=\lr{A'(u)}{T'(0)u}=-\lr{Bu}{JA'(u)}.
\end{eqnarray*}
Therefore, formally, we have
\begin{eqnarray*}
\frac{d}{dt}A(u(t))=\lr{A'(u)}{u_t}=\lr{A'(u)}{\tilde{J}E'(u)}=-\lr{E'(u)}{JA'(u)}=-P(u).
\end{eqnarray*}
By Lemma $4.6$ of $\cite{GSS}$, we have $A\circ u \in C^1$ for $u\in C({\cal I},X)\cap C^1({\cal I},Y^*)$.
Therefore, the formal calculation is justified.
\end{Rem}

\begin{Rem}
$A$ and $P$ are invariant under $T$, that is
\begin{eqnarray*}
A(T(s)u)&=&A(u),\\
P(T(s)u)&=&P(u).
\end{eqnarray*}
Indeed, the invariance of $A$ follows from the invariance of $M$ and $\La$.
The invariance of $P$ follows from the invariance of $S$ and $A$.
More precisely, since
$A(T(s)u+h)=A(u+T(-s)h)$, we have $A'(T(s)u)=T(s)A'(u)$.
So, we have
\begin{eqnarray*}
P(T(s)u)=\lr{S'(T(s)u)}{JA'(T(s)u)}=\lr{T(s)S'(u)}{JT(s)A'(u)}=P(u),
\end{eqnarray*}
where we have used the fact $J$ and $T(s)$ commutes.
\end{Rem}

We now calculate the value of $P$ along the curve $\Psi$.
\begin{Lem}
Let $\ep>0$ sufficiently small.
Assume $\lr{\pom}{J^{-1}\pd_{\om}^2\pom}=0$.
Then, for $|\la|<\ep$, we have
\begin{eqnarray*}
P(\Psi(\la))=\eta_2(\la)+o(\eta_2(\la)),\ \la\rightarrow 0.
\end{eqnarray*}
\end{Lem}

\begin{proof}
First, we calculate $S_{\om+\La(\Psi(\la))}'(\Psi(\la))$.
\begin{eqnarray*}
S_{\om+\La(\Psi(\la))}'(\Psi(\la))&=&S_{\om+\la}'(\phi_{\om+\la}+\sigma(\la)B\lpom{\la})\\
&=&\sigma(\la)S_{\om+\la}''(\lpom{\la})B\lpom{\la}+o(\sigma(\la)).
\end{eqnarray*}
Next, we calculate $JA'(\Psi(\la))$.
Recall that $M(\Psi(\la))=\Psi(\la)=\lpom{\la}+\sigma(\la)B\lpom{\la}$ and we assumed \begin{eqnarray*}
-\lr{B\pom}{\ppom}=d''(\om)=0,
\end{eqnarray*}
and $\lr{\pom}{J^{-1}\pd_{\om}^2\pom}=0$.
So, we have
\begin{eqnarray*}
\lr{T'(0)M(\Psi(\la))}{J^{-1}\lppom{\la}}&=&o(1),\ \la\rightarrow 0,\\
\lr{M(\Psi(\la))}{J^{-1}\pd_{\om}^2\phi_{\om+\la}}&=&o(1),\ \la\rightarrow 0.
\end{eqnarray*}
Therefore, by (\ref{eq:A'}), we have
\begin{eqnarray*}
JA'(\Psi(\la))=\lppom{\la}+o(1),\ \la\rightarrow 0.
\end{eqnarray*}
Combining these calculations, we have
\begin{eqnarray*}
P(\Psi(\la))&=&\sigma(\la)\lr{S_{\om+\la}''(\lpom{\la})B\lpom{\la}}{\lppom{\la}}+o(\sigma(\la))\\
&=&\sigma(\la)||B\lpom{\la}||_{H}+o(\sigma(\la))\\
&=&\eta_2(\la)+o(\eta_2(\la)),\ \la\rightarrow 0,
\end{eqnarray*}
where we have used the relation $S_{\om+\la}''(\lpom{\la})\lppom{\la}=B\lpom{\la}$.
\end{proof}

The following lemma is well known.
For example see \cite{Ohta11} Lemma 7.

\begin{Lem}\label{lem:pos}
There exists $k_0>0$ such that if $w\in X$ satisfy $\lr{w}{\Tpom}=\lr{w}{\ppom}=\lr{w}{B\pom}=0$,
then $\lr{\Hom w}{w}\geq k_0||w||_X^2$.
\end{Lem}

By a continuity argument and Lemma \ref{lem:pos}, we can show the following lemma.

\begin{Lem}\label{lem:pos2}
There exists $\ep_0>0$ such that for $|\la|<\ep_0$, if $w\in X$ satisfies $\lr{w}{T'(0)\phi_{\om+\la}}=\lr{w}{\lppom{\la}}=\lr{w}{B\lpom{\la}}=0$,
then $\lr{\Hom w}{w}\geq \half k_0||w||_X^2$.
\end{Lem}

\section{Proof of Theorems \ref{thm:stability} and \ref{thm:instability}}\label{sec:pt}
In this section we prove Theorems \ref{thm:stability} and \ref{thm:instability}.
As section \ref{sec:p}, we assume Assumptions \ref{ass:sol}, \ref{ass:bddst}, \ref{ass:spec}, (\ref{eq:comparable}) and $d''(\om)=0$.
We first estimate $\al(u)$ which is given in Lemma \ref{lem:modul}.

\begin{Lem}\label{lem:est_al}
Let $\ep>0$ sufficiently small.
Let $u\in N_{\ep}^0$.
Let $\sigma$ as in Lemma $\ref{lem:defsigma}$ and $\al(u)$, $w(u)$ and $\La$ as in Lemma $\ref{lem:modul}$.
Then, there exists a constant $C>0$ such that
\begin{eqnarray*}
|\al(u)|\leq C(\sigma(\La(u))||w(u)||_X+||w(u)||_X^2).
\end{eqnarray*}
\end{Lem}

\begin{proof}
We first calculate $Q(u)$. By Lemma \ref{lem:modul} and (\ref{eq:defQ}) (definition of $Q$), we have
\begin{eqnarray*}
Q(\pom)&=& Q(u)\\
&=& Q(\Psi(\La(u))+w(u)+\al(u)B\lpom{\La(u)})\\
&=& Q(\Psi(\La(u)))+Q(w(u)+\al(u)B\lpom{\La(u)})\\&&+\lr{B\lpom{\La(u)}+\sigma(\La(u))B^2\lpom{\La(u)}}{w(u)+\al(u)B\lpom{\La(u)}}\\
&=&Q(\pom)+\al(u)||B\lpom{\La(u)}||_H^2+\sigma(\La(u))\lr{B^2\lpom{\La(u)}}{w(u)}\\&&+\al(u)\sigma(\La(u))\lr{B^2\lpom{\La(u)}}{B\lpom{\La(u)}}+Q(w(u))\\&&+\al(u)\lr{Bw(u)}{B\lpom{\La(u)}}+\al(u)^2Q(B\lpom{\La(u)}).
\end{eqnarray*}
Therefore, we have
\begin{eqnarray*}
-\al(u)\left(||B\pom||_H^2+o(1)\right)=\sigma(\La(u))\lr{B^2\lpom{\La(u)}}{w(u)}+Q(w(u)), \La(u)\rightarrow 0.
\end{eqnarray*}
Thus, we have the conclusion.
\end{proof}

Next, we show that under a restriction of the value of $S_{\om}$ we get a good estimate for $w(u)$ and $\al(u)$.

\begin{Lem}\label{lem:est_w}
Let $\ep>0$ sufficiently small.
Let $a\in\R$.
Suppose $u\in N_{\ep}^0$ and 
\begin{eqnarray*}
S_{\om}(u)-S_{\om}(\pom)\leq  a\eta_1(\La(u))+\frac{k_0}{10}||w(u)||_X^2,
\end{eqnarray*}
where $k_0$ is given in Lemma \ref{lem:pos}.
Then, $||w(u)||_X^2=O(\eta_1(\La(u)))$ as $\La(u)\rightarrow 0$.
In particular, $\al(u)=O(\eta_1(\La(u)))$ as $\La(u)\rightarrow 0$.
\end{Lem}

\begin{proof}
Suppose there exists $u_n\in N_{\ep}^0$, $u_n\rightarrow\pom$ in $X$, s.t. \begin{eqnarray*}
S_{\om}(u_n)-S_{\om}(\pom)\leq  a\eta_1(\La_n)+\frac{k_0}{10}||w_n||_X^2,
\end{eqnarray*}
and $||w_n||_X^2=C_n\eta_1(\La_n)$,
where $w_n=w(u_n)$, $\La_n=\La(u_n)$, $\al_n=\al(u_n)$ and $C_n\rightarrow\infty$.
Then, we have $\eta_1(\La_n)=o(||w_n||_X^2)$.
Further, by Lemma \ref{lem:defsigma}, (\ref{eq:comparable}) and assumption of contraditicon, we have
\begin{eqnarray*}
\sigma(\La_n)\sim\eta_2(\La_n)\sim\frac{\eta_1(\La_n)}{\La_n}=\frac{||w_n||_X^2}{\La_nC_n}=\frac{\eta_1^{1/2}(\La_n)}{\La_nC_n^{1/2}}||w_n||_X=o(||w_n||_X),\ n\rightarrow\infty,
\end{eqnarray*}
where we have used 
``$\sim$'' in the sense of $(\ref{eq:defsim})$.
Thus, by Lemma \ref{lem:est_al}, $\al_n=O(||w_n||_X^2)$.
Now, by Taylor expansion and Lemma \ref{lem:modul},
\begin{eqnarray*}
\Som(u_n)-\Som(\pom)&=&\Som(\Psi(\La_n)+w_n+\al_nB\lpom{\La_n})-\Som(\pom)\\
&=& \Som(\Psi(\La_n))-\Som(\pom)+\lr{\Som'(\Psi(\La_n))}{w_n+\al_nB\lpom{\La_n}}\\&&+\half\lr{\Som''(\Psi(\La_n))w_n}{w_n}+o(||w_n||_X^2),\ n\rightarrow\infty.
\end{eqnarray*}
Further, by Lemma \ref{lem:S} and $S_{\om}'(\pom)=0$, we have
\begin{eqnarray*}
\Som(\Psi(\La_n))-\Som(\pom)&=&O(\eta_1(\La_n))=o(||w_n||_X^2),\\
\lr{\Som'(\Psi(\La_n))}{\al_nB\lpom{\La_n}}&=&o(||w_n||_X^2),\ n\rightarrow\infty,
\end{eqnarray*}
and by $S_{\om}'=S_{\om+\la}'+\la B$, $\lr{B\lpom{\La_n}}{w_n}=0$ and $\sigma(\La_n)=o(||w_n||_X)$ as $n\rightarrow\infty$, we have
\begin{eqnarray*}
\lr{\Som'(\Psi(\La_n))}{w_n}&=&\lr{S_{\om+\La_n}'(\Psi(\La_n))+B\Psi(\La_n)}{w_n}\\
&=&\lr{S_{\om+\La_n}'(\Psi(\La_n))+\sigma(\La_n)B\lpom{\La_n}}{w_n}\\
&=&o(||w_n||_X^2),\ n\rightarrow\infty.
\end{eqnarray*}
Therefore, by Lemma \ref{lem:pos2}, we have
\begin{eqnarray*}
\Som(u_n)-\Som(\pom)&=&\half\lr{\Som''(\Psi(\La_n))w_n}{w_n}+o(||w_n||_X^2)\\
&\geq &\frac{k_0}{4}||w_n||_X^2+o(||w_n||_X^2)\\
&\geq &\frac{k_0}{8}||w_n||_X^2,
\end{eqnarray*}
for sufficiently large $n$.
This contradicts to the assumption.
Therefore, we have the conclusion.
\end{proof}

\begin{proof}[Proof of Theorem \ref{thm:stability}]
Let $u\in N_{\ep}^0$.
Suppose, $S_{\om}(u)-S_{\om}(\pom)< \eta_1(\La(u))+\frac{k_0}{10}||w(u)||_X^2$.
Then, by Lemma \ref{lem:est_w}, we have $||w(u)||_X^2=O(\eta_1(u))$, $\La(u)\rightarrow 0$.
Now,
\begin{eqnarray*}
\Som(u)-\Som(\pom)&=&\Som(\Psi(\La(u))+w(u)+\al(u)B\lpom{\La(u)})\\
&=&\eta_1(\La(u))+\lr{\Som'(\Psi(\La(u)))}{w(u)}+\half\lr{\Som''(\pom)w(u)}{w(u)}\\&&+o(\eta_1(\La(u))).
\end{eqnarray*}
Using $S_{\om}'=S_{\om+\la}'+B$, $\sigma(\La(u))=O(\eta_2(\La(u)))$ and $||w(u)||_X=O(\eta_1(\La(u))^{1/2})$,
\begin{eqnarray*}
\lr{\Som'(\Psi(\La(u)))}{w(u)}&=&\lr{S_{\om+\La(u)}'(\Psi(\La(u)))+\La(u)B\Psi(\La(u))}{w(u)}\\
&=&\sigma(\La(u))\lr{S_{\om}''(\lpom{\La(u)})B\lpom{\La(u)}}{w(u)}\\&&+\La(u)\sigma(\La(u))\lr{B^2\pom}{w(u)}\\
&=&o(\eta_1(\La(u))).
\end{eqnarray*}
Since we have assumed that $d$ is strictly convex in an open neighborhood of $\om$, $\eta_2(\la)$ is strictly increasing in an open neighborhood of $0$ (if $\eta_2(\la)$ is not increasing, then $d$ would not be convex, if $\eta_2(\la)$ is constant, then $d$ would not be strictly convex).
So, we have
\begin{eqnarray*}
\Som(u)-\Som(\pom)\geq c\La(u)\eta_2(\La(u))+\frac{k_0}{4}||w||_X^2,
\end{eqnarray*}
for a constant $c>0$.

Now, suppose that there exists a sequence of solutions $u_n$, and $t_n>0$ s.t. $u_n\rightarrow \pom$ in $X$ and $\inf_{s\in\R}||u_n(t_n)-T(s)\pom||_X=\ep_0/10$.
Take
\begin{eqnarray*}
v_n:=\sqrt{Q(\pom)/Q(u_n)}u_n(t_n).
\end{eqnarray*}
Since $\sqrt{Q(\pom)/Q(u_n)}\rightarrow 1$, we have $||v_n-u_n(t_n)||_X\rightarrow 0$ and $S_{\om}(v_n)-S_{\om}(\pom)\rightarrow 0$.
Thus, $\La(v_n)$, $w(v_n)$ and $\al(v_n)$ converges to zero.
This implies 
\begin{eqnarray*}
\inf_{s\in\R}||u_n(t_n)-T(s)\pom||_X\rightarrow 0.
\end{eqnarray*}
This is a contradiction.
\end{proof}

We next show Theorem \ref{thm:instability}.
We first calculate $P$.

\begin{Lem}\label{lem:P}
Let $\ep>0$, sufficiently small.
Let $u\in N_{\ep}^0$ and $\Som(u)-\Som(\pom)<0$.
Further, assume $\lr{\lppom{\La(u)}}{J^{-1}\pd_{\om}\lpom{\La(u)}}=0$.
Then
\begin{eqnarray*}
P(u)=\eta_2(\La(u))+o(\eta_2(\La(u))).
\end{eqnarray*}
\end{Lem}

\begin{proof}
By Taylor expansion,
\begin{eqnarray*}
P(u)&=&P(\Psi(\La(u))+w(u))+o(\eta_2(\La(u)))\\
&=&\eta_2(\La(u))+\lr{S_{\om+\La(u)}''(\Psi(\La(u)))w(u)}{JA'(\Psi(\La(u)))}\\&&+\lr{S_{\om+\La(u)}'(\Psi(\La(u)))}{JA''(\Psi(\La(u)))w(u)}+o(\eta_2(\La(u)))\\
&=&\eta_2(\La(u))+\lr{S_{\om+\La(u)}''(\Psi(\La(u)))w(u)}{JA'(\Psi(\La(u)))}\\&&+o(\eta_2(\La(u))),\ \La(u)\rightarrow 0,
\end{eqnarray*}
where we have used $||w(u)||_X^2=o(\eta_2(\La(u)))$ and $S_{\om+\La(u)}'(\Psi(\La(u)))=O(\eta_2(\La(u)))$.
Now, by (\ref{eq:A'}),
\begin{eqnarray*}
JA'(\Psi(\La(u)))&=&\lppom{\La(u)}-\lr{B\lpom{\La(u)}}{\lppom{\La(u)}}\theta'(\Psi(\La(u)))\\&&+\lr{\lppom{\La(u)}}{J^{-1}\pd_{\om}\lpom{\La(u)}}\La'(\Psi(\La(u)))\\&&+O(\eta_2(\La(u))), \ \La(u)\rightarrow 0,
\end{eqnarray*}
where we have used Lemma \ref{lem:defsigma}.
Now, by $\lr{\lppom{\La(u)}}{J^{-1}\pd_{\om}\lpom{\La(u)}}=0$, 
$\lr{w(u)}{B\lpom{\La(u)}}=0$ and
 (\ref{eq:SB}),
 we have
\begin{eqnarray*}
&&\lr{S_{\om+\La(u)}''(\Psi(\La(u)))w(u)}{JA'(\Psi(\La(u)))}=\lr{S_{\om+\La(u)}''(\lpom{\La(u)})w(u)}{\lppom{\La(u)}}\\&&-\lr{B\lpom{\La(u)}}{\lppom{\La(u)}}\lr{w(u)}{S_{\om+\La(u)}''(\lpom{\La(u)})\theta'(\Psi(\La(u)))}+o(\eta_2(\La(u)))\\
&=&o(\eta_2(\La(u))),\ \La(u)\rightarrow0,
\end{eqnarray*}
where we have used $||w||_X=O(\eta_1(\La(u))^{1/2})$ and $\theta'(\Psi(\La(u)))$ is a linear combination of $\lppom{\La(u)}$ and $T'(0)\lpom{\La(u)}$ because of (\ref{eq:G'}).
Therefore, we have the conclusion.
\end{proof}

\begin{proof}[Proof of Theorem \ref{thm:instability}]
By the assumption of Theorem \ref{thm:instability}, we have $\eta_1(\la)<0$ in a one-sided open neighborhood of $0$.
Therefore, by Lemma \ref{lem:S}, we can take the initial data from $\Psi(\la_n)$, where $S(\Psi(\la_n))<S(\pom)$ and $\la_n\rightarrow 0$.
Suppose, $u_n$ stays in $N_{\ep}^0$.
By the conservation of $E$ and $Q$, we have 
\begin{eqnarray*}
\Som(u_n(t))-\Som(\pom)=\eta_1(\La(u_n(t)))+o(\eta_1(\La(u_n(t)))),
\end{eqnarray*}
and by Lemma \ref{lem:P},
\begin{eqnarray*}
P(u(t))=\eta_2(\La(u_n(t)))+o(\eta_2(\La(u_n(t)))).
\end{eqnarray*}
Then, since $\la\eta_2(\la)\sim \eta_1(\la)$, we have
\begin{eqnarray*}
\Som(\pom)-\Som(u_n(t))\leq C|\La(u_n(t)) P(u_n(t))|,
\end{eqnarray*}
for some constant $C>0$.
Thus, we have $0<\delta<|P(u_n(t))|$ for arbitrary $t$.
So, $P$ has same sign.
Suppose $P>0$.
Then, $\frac{dA}{dt}(u_n(t))>P(u_n(t))>\delta$.
Thus, $A$ is unbounded.
However, this is contradiction.
For the case $P<0$ we have the same conclusion.
\end{proof}

\section{Examples}\label{sec:e}
\subsection{The nonlinear Klein-Gordon equations}
We consider the following single power nonlinear Klein-Gordon (NLKG) equation.
\begin{eqnarray}\label{eq:nlkg}
u_{tt}-\De u +u -|u|^{p-1}u=0, (x,t)\in\R^d,
\end{eqnarray}
where $d\geq 1$ and $1<p<\infty$ for $d=1,2$ and $1<p<1+4/(d-2)$ for $d\geq 3$.
To put (\ref{eq:nlkg}) on to our setting, set
$X=H_r^1(\R^d)\times L_r^2(\R^d)$, $Y=L^2_r(\R^d)\times H^1_r(\R^d)$ and $H=(L_r^2(\R^d))^2$, where $H^1_r$ and $L^2_r$ are subspaces of $H^1$ and $L^2$ which consists with radial functions.
Then define $J$ and $E$ as
\begin{eqnarray*}
J&=&\left(
  \begin{matrix}
        0 & 1 \cr
        -1 & 0 \cr
  \end{matrix}
\right),\\
E(U)&=&\half\int |v|^2+|\nabla u|^2+|u|^2-\frac{1}{p+1}\int |u|^{p+1}.
\end{eqnarray*}
Then, $J:H\rightarrow H$ is skew symmetric, and $J|_X:X\rightarrow Y$, $J|_Y:Y\rightarrow X$ are bounded and bijective.
Furhter, $E$ is $C^2$.
Let $U=(u,v)^t$, where $t$ means transposition.
Then NLKG equation is rewritten as
\begin{eqnarray*}
\frac{d}{dt}U=JE'(U)
\end{eqnarray*}
in $Y^*$.
Further, in this case, we take $T(s)=e^{is}I$, where $I$ is the identity matrix.
So, we have
$Q(u)=\mathrm{Im}\int \bar{u}u_t$.
From the results of Ginibre and Velo \cite{GV85}, it is known that NLKG equation is locally well-posed and $E$ and $Q$ are conserved (i.e.\ Assumption \ref{ass:sol} is satisfied).
For, $\om^2<1$, let $\pom$ be the unique positive radial solution of
\begin{eqnarray*}
0=-\De \pom +(1-\om^2)\pom-\pom^p.
\end{eqnarray*}
Then, $e^{i\om t}\pom$ is the solution of (\ref{eq:nlkg}).
It is well known that $\phi\in {\cal S}(\R^d)$, where ${\cal S}(\R^d)$ is the Schwartz space (see for example Chapter B of \cite{TaoBook}).
Further, by scaling, we have $\phi_{\om}=(1-\om^2)^{1/(p-1)}\phi_0((1-\om^2)^{1/2}x)$.
Therefore, it is easy to check $\om \mapsto \pom$ satisfies Assumption \ref{ass:bddst}.
Further, Assumption \ref{ass:spec} is also well known to be satisfied (see for example \cite{Weinstein85}).


Now, since $\phi_{\om}=(1-\om^2)^{1/(p-1)}\phi_0((1-\om^2)^{1/2}x)$, we can calculate $d$ directly.
Since $Q(\pom)=\om\int\pom^2$, we have
\begin{eqnarray*}
d''(\om)=-\left(1-(1+\frac{4}{p-1}-d)\om^2\right)(1-\om^2)^{\frac{2}{p-1}-\frac{d}{2}-1}\int_{\R^d}\phi_0^2.
\end{eqnarray*}
So, we see that for the case $p>1+4/d$, then $d''(\om)<0$ for all $\om\in(-1,1)$ and for the case $1<p<1+4/d$, there exists
\begin{eqnarray*}
0<\om_*=\sqrt{\frac{p-1}{4-(d-1)(p-1)}}<1,
\end{eqnarray*}
such that if $|\om|<\om_*$, then $d''(\om)<0$ and if $|\om|>\om_*$, then $d''(\om)>0$.
Therefore, in these case we know the stability and instability.
These are the results by \cite{Shatah83} and \cite{ShatahStrauss85}.

For the case $\om=\pm\om_*$, we can show $d'''(\om_*)\neq 0$, so by Corollary \ref{cor:instability}, we see that in this case, we have the instability.

We have to remark that for the case $d\geq 2$, this result was proved by Ohta and Todorova \cite{OT07} and for the case $d=1$, $p\geq 2$, one can prove this result by applying Comech and Pelinovsky's result \cite{CP03} (for the case $1<p<2$, it seems that the Assumption 2.10 (b) of \cite{CP03} is not satisfied).
Therefore, for $1<p<2$, $d=1$, this result seems to be new.
Further, our proof, the proof of \cite{OT07} and the proof of \cite{CP03} are completely different from each other and our proof gives a simple and unified proof for the critical case.

\subsection{The nonlinear Schr\"odinger equations}
We next consider the double power nonlinear Schr\"odinger equations.
\begin{eqnarray*}
iu_t+\pd_{xx}u+a_1|u|^{p_1-1}u+a_2|u|^{p_2-1}u, (t,x)\in\R^2,
\end{eqnarray*}
where $a_1, a_2\in\R$ and $1<p_1<p_2<\infty$.
In this case, let $X=Y=H^1_r(\R)$, $H=L^2_r(\R)$, $J=i$, $T(s)=e^{is}$ and
\begin{eqnarray*}
E(u)=\half\int_{\R}|u_x|^2\,dx-\frac{a_1}{p_1+1}\int_{\R}|u|^{p_1+1}\,dx-\frac{a_1}{p_1+1}\int_{\R}|u|^{p_1+1}\,dx.
\end{eqnarray*}
Then, we are on the setting of our theory.
In this case, by the combination of $a_1$, $a_2$, it is known that there exists some $\om>0$ such that $d''(\om)=0$ and $d'''(\om)\neq 0$ (See \cite{M08}).
So, for such $\om>0$, we can show the instability.

\vglue 1\baselineskip
\noindent
{\bf Acknowledgments.}
The author wants to thank the helpful discussions with Professor Masahito Ohta.
The author would also like to express his deep gratitude to Professor Yoshio Tsutsumi for his helpful comments.
The author was partially supported by Grant-in-Aid for JSPS Fellows (20$\cdot$56371).


\begin{thebibliography}{10}

\bibitem{CP03}
Andrew Comech and Dmitry Pelinovsky, \emph{Purely nonlinear instability of
  standing waves with minimal energy}, Comm. Pure Appl. Math. \textbf{56}
  (2003), no.~11, 1565--1607. \MR{1995870 (2005h:37176)}

\bibitem{GV85}
J.~Ginibre and G.~Velo, \emph{The global {C}auchy problem for the nonlinear
  {K}lein-{G}ordon equation}, Math. Z. \textbf{189} (1985), no.~4, 487--505.
  \MR{786279 (86f:35149)}

\bibitem{GSS}
Manoussos Grillakis, Jalal Shatah, and Walter Strauss, \emph{Stability theory
  of solitary waves in the presence of symmetry. {I}}, J. Funct. Anal.
  \textbf{74} (1987), no.~1, 160--197. \MR{MR901236 (88g:35169)}

\bibitem{M08}
Masaya Maeda, \emph{Stability and instability of standing waves for
  1-dimensional nonlinear {S}chr\"odinger equation with multiple-power
  nonlinearity}, Kodai Math. J. \textbf{31} (2008), no.~2, 263--271.
  \MR{2435895 (2009k:35304)}

\bibitem{M10}
\bysame, \emph{Instability of bound states of nonlinear {S}chr\"odinger
  equations with {M}orse index equal to two}, Nonlinear Anal. \textbf{72}
  (2010), no.~3-4, 2100--2113. \MR{2577607 (2010k:35466)}

\bibitem{Ohta11}
Masahito Ohta, \emph{Instability of bound states for abstract nonlinear
  {S}chr\"odinger equations}, J. Funct. Anal. \textbf{261} (2011), no.~1,
  90--110.

\bibitem{OT07}
Masahito Ohta and Grozdena Todorova, \emph{Strong instability of standing waves
  for the nonlinear {K}lein-{G}ordon equation and the
  {K}lein-{G}ordon-{Z}akharov system}, SIAM J. Math. Anal. \textbf{38} (2007),
  no.~6, 1912--1931 (electronic). \MR{2299435 (2008a:35198)}

\bibitem{Shatah83}
Jalal Shatah, \emph{Stable standing waves of nonlinear {K}lein-{G}ordon
  equations}, Comm. Math. Phys. \textbf{91} (1983), no.~3, 313--327. \MR{723756
  (84m:35111)}

\bibitem{ShatahStrauss85}
Jalal Shatah and Walter Strauss, \emph{Instability of nonlinear bound states},
  Comm. Math. Phys. \textbf{100} (1985), no.~2, 173--190. \MR{MR804458
  (87b:35159)}

\bibitem{TaoBook}
Terence Tao, \emph{Nonlinear dispersive equations}, CBMS Regional Conference
  Series in Mathematics, vol. 106, Published for the Conference Board of the
  Mathematical Sciences, Washington, DC, 2006, Local and global analysis.
  \MR{2233925 (2008i:35211)}

\bibitem{Weinstein85}
Michael~I. Weinstein, \emph{Modulational stability of ground states of
  nonlinear {S}chr\"odinger equations}, SIAM J. Math. Anal. \textbf{16} (1985),
  no.~3, 472--491. \MR{MR783974 (86i:35130)}

\end{thebibliography}

\def\cprime{$'$}
\providecommand{\bysame}{\leavevmode\hbox to3em{\hrulefill}\thinspace}
\providecommand{\MR}{\relax\ifhmode\unskip\space\fi MR }
\providecommand{\MRhref}[2]{%
  \href{http://www.ams.org/mathscinet-getitem?mr=#1}{#2}
}
\providecommand{\href}[2]{#2}

\end{document}